\documentclass{amsart}
\usepackage{bm}
\usepackage{amssymb}
\usepackage{amsthm}
\usepackage{mathtools}
\usepackage{subfiles}
\usepackage[shortlabels]{enumitem}
\usepackage{float}
\usepackage{tikz}
\usetikzlibrary{backgrounds}
\usetikzlibrary{arrows}

\usetikzlibrary{arrows.meta}
\usetikzlibrary{shapes,shapes.geometric,shapes.misc}
\usetikzlibrary{fit}
\usetikzlibrary{positioning}
\usetikzlibrary{calc}
\tikzstyle{tikzfig}=[baseline=-0.25em,scale=0.5]
\pgfdeclarelayer{edgelayer}
\pgfdeclarelayer{nodelayer}
\pgfsetlayers{background,nodelayer,edgelayer,main}
\tikzstyle{none}=[inner sep=0mm]
\tikzstyle{empty dot}=[fill=none, draw=black, shape=circle]
\tikzstyle{every loop}=[]
\tikzset{>={latex[width=1mm,length=1mm]}}
\newcommand{\drawThm}[3]{ \draw #1 %
	\ifx&#2&%
	\else
	node[style=thmref, pos=0.01] {\resizebox{5mm}{3mm}{#2}} %
	\fi
	\ifx&#3&%
	\else
	node[style=thmref, pos=0.99] {\resizebox{5mm}{3mm}{#3}}
	\fi ;
}

\tikzstyle{box}=[fill={rgb,255: red,228; green,228; blue,228}, draw=black, shape=rectangle]

\tikzstyle{reducible}=[->, dashed]
\tikzstyle{strictreducible}=[->]
\tikzstyle{nonreducible}=[->, draw=red]
\tikzstyle{uncomp}=[draw=red, <->]
\usepackage{tikz-cd}
\usepackage{bbm} 
\usepackage{hyperref}
\hypersetup{
	colorlinks=true,
	linkcolor=red,
	filecolor=magenta,      
	urlcolor=cyan,
	pdftitle={Overleaf Example},
	pdfpagemode=FullScreen,
}
\usepackage[capitalise]{cleveref}
\usepackage{verbatim}
\makeatletter
\newtheorem*{rep@theorem}{\rep@title}
\newcommand{\newreptheorem}[2]{%
	\newenvironment{rep#1}[1]{%
		\def\rep@title{#2 \ref{##1}}%
		\begin{rep@theorem}}%
		{\end{rep@theorem}}}
\makeatother

\newtheorem{theorem}{Theorem}[section]
\newreptheorem{theorem}{Theorem}
\newtheorem{corollary}[theorem]{Corollary}
\newreptheorem{corollary}{Corollary}
\newtheorem{lemma}[theorem]{Lemma}

\newtheorem{proposition}[theorem]{Proposition}
\newreptheorem{proposition}{Proposition}

\newtheorem*{claimn*}{Claim}

\theoremstyle{definition}
\newtheorem{definition}[theorem]{Definition} 
\newtheorem{remark}[theorem]{Remark}

\newcommand{\cantor}{2^{\omega}}
\newcommand{\baire}{\omega^{\omega}}
\newcommand{\ramsey}{[\omega]^{\omega}}
\newcommand{\bSigma}[2]{\bm{\Sigma}^{#1}_{#2}}
\newcommand{\bPi}[2]{\bm{\Pi}^{#1}_{#2}}
\newcommand{\bDelta}[2]{\bm{\Delta}^{#1}_{#2}}
\newcommand{\lSigma}[2]{\Sigma^{#1}_{#2}}
\newcommand{\lPi}[2]{\Pi^{#1}_{#2}}
\newcommand{\lDelta}[2]{\Delta^{#1}_{#2}}
\newcommand{\tr}{\boldsymbol{\mathrm{Tr}}}

\newcommand{\open}[1]{[#1]^{\prec}}
\newcommand{\height}{\mathrm{ht}}

\newcommand{\concat}[2]{#1^{\smallfrown}#2}
\newcommand{\stem}[1]{\mathrm{stem}(#1)}

\newcommand{\pf}[1]{\mathrm{pf}(#1)}

\newcommand{\T}{\mathrm{T}}

\newcommand{\ari}{\mathrm{a}}

\newcommand{\ock}{\omega^{\mathrm{CK}}_1}
\newcommand{\ko}{\mathcal{O}}

\newcommand{\multif}{\rightrightarrows}
\newcommand{\subc}{:\subseteq}

\newcommand{\W}{\mathrm{W}}
\newcommand{\sW}{\mathrm{sW}}
\newcommand{\lew}{\leq_{\W}}		 
\newcommand{\lews}{\leq_{\sW}}
\newcommand{\tCCantor}{\mathsf{TC}_{\cantor}}
\newcommand{\CCantor}{\mathsf{C}_{\cantor}}
\newcommand{\tCBaire}{\mathsf{TC}_{\baire}}
\newcommand{\CBaire}{\mathsf{C}_{\baire}}
\newcommand{\stCCantor}{\mathsf{sTC}_{\cantor}}
\newcommand{\stCBaire}{\mathsf{sTC}_{\baire}}
\newcommand{\UCBaire}{\mathsf{UC}_{\omega^{\omega}}}
\newcommand{\J}{\mathsf{J}}
\newcommand{\leftof}[2]{L ({#1}, {#2})}

\newcommand{\wfind}{\mathsf{wFindHS}}

\newcommand{\SRT}[1]{\bSigma{0}{#1}\text{-}\mathsf{RT}}
\newcommand{\DRT}[1]{\bDelta{0}{#1}\text{-}\mathsf{RT}}

\newcommand{\atrw}{\ensuremath{\mathsf{ATR}}}
\newcommand{\rca}{\ensuremath{\mathsf{RCA}_0}}
\newcommand{\wkl}{\ensuremath{\mathsf{WKL}_0}}
\newcommand{\aca}{\ensuremath{\mathsf{ACA}_0}}
\newcommand{\atr}{\ensuremath{\mathsf{ATR}_0}}

\DeclareMathOperator{\dom}{dom}
\DeclareMathOperator{\ran}{ran}
\newcommand{\upto}{{\upharpoonright}}
\newcommand{\id}[1]{\mathsf{id}_#1}

\newcommand{\N}{\ensuremath{\mathbb{N}}}
\newcommand{\zfc}{\ensuremath{\mathsf{ZFC}}}

\newcommand{\einf}{\exists^\infty}
\newcommand{\finf}{\forall^\infty}
\newcommand{\wo}{\mathrm{WO}}

\newcommand{\lld}{\mathsf{LL}_{\bDelta{0}{1}}}
\newcommand{\hld}{\mathsf{HL}_{\bDelta{0}{1}}}
\newcommand{\lls}{\mathsf{LL}_{\bSigma{0}{1}}}
\newcommand{\hls}{\mathsf{HL}_{\bSigma{0}{1}}}
\newcommand{\hlp}{\mathsf{HL}_{\bPi{0}{1}}}
\newcommand{\wlinds}{\mathsf{wFindL}_{\bSigma{0}{1}}}
\newcommand{\wlindp}{\mathsf{wFindL}_{\bPi{0}{1}}}
\newcommand{\wlindd}{\mathsf{wFindL}_{\bDelta{0}{1}}}
\newcommand{\linds}{\mathsf{FindL}_{\bSigma{0}{1}}}
\newcommand{\lindp}{\mathsf{FindL}_{\bPi{0}{1}}}
\newcommand{\lindd}{\mathsf{FindL}_{\bDelta{0}{1}}}
\newcommand{\LL}{\mathrm{LT}}
\title{Computable Aspects of the Laver Partition Theorem}
\author{Alberto Marcone}
\address
	{Dipartimento di Scienze Matematiche, Informatiche e Fisiche\\
		Universit\`a di Udine\\
		33100 Udine\\
		Italy}
	\email{\href{mailto:alberto.marcone@uniud.it}{alberto.marcone@uniud.it}}
\author{Gian Marco Osso}
\address
	{Dipartimento di Scienze Matematiche, Informatiche e Fisiche\\
		Universit\`a di Udine\\
		33100 Udine\\
		Italy}
	\email{\href{mailto:osso.gianmarco@spes.uniud.it}{osso.gianmarco@spes.uniud.it}}
\date{\today}
\thanks{Marcone is a member of INdAM-GNSAGA}

\begin{document}

\begin{abstract}
    The Laver Partition Theorem is a fundamental tool in the analysis of Laver and Hechler forcings. It is also connected to determinacy and the Galvin-Prikry theorem: indeed it can be seen as the common core of these two theorems. We study the reverse mathematics and Weihrauch degrees of the Laver Partition Theorem restricted to open and clopen sets. We obtain upper and lower bounds on the proof theoretic strength of this result, as well as a precise picture of the (arithmetical) Weihrauch degrees of the problems related to it.
\end{abstract}

\maketitle

	\section{Introduction}

    In \cite{Laver76} Richard Laver introduced a forcing notion whose conditions are trees that turned out to be very important.
	Recall that a \emph{Laver tree} $T \subseteq \omega^{<\omega}$ is a tree with a string $\sigma$ (the \textit{stem} of $T$) such that every string in $T$ is compatible with $\sigma$ and every extension of $\sigma$ in $T$ has infinitely many immediate successors. In Laver forcing, conditions are ordered by inclusion.
    
    Our goal is to investigate the proof and computability theoretic strength of a fragment of the Laver Partition Theorem, which states that for every Borel set $B \subseteq \baire$ and every Laver tree $T$ there is a Laver tree $T' \subseteq T$ with the same stem as $T$ and such that either $[T'] \subseteq B$ or $[T'] \cap B= \emptyset$. The statement above is easily seen to follow from both the Galvin-Prikry Theorem and Borel Determinacy, and in fact these implications hold level-by-level (e.g.\ $\bSigma{0}{\alpha}$-$\mathsf{DET}$ implies the Laver Partition Theorem for $\bSigma{0}{\alpha}$ subsets of Baire space). 
	
	We note that the Laver Partition Theorem and the Galvin-Prikry Theorem can be understood as saying, in set theoretic language, that Laver (resp.\ Mathias) forcing has \emph{the pure decision property}: given any Borel set $A$ and any condition $\bm{p}$ in the appropriate forcing notion, we can extend $\bm{p}$ without extending the ``finitary part'' (by which we mean the stem of the tree, in the case of Laver forcing, and the finite string making up half of a Mathias condition in the case of Mathias forcing) in such a way to \emph{strongly} force either in or out of $A$.

	In this paper, we focus on the Laver Partition Theorem restricted to either open or clopen sets, which we henceforth refer to as $\lls$ and $\lld$. It is not hard to see that, since the part of a Laver tree above the stem is naturally isomorphic to $\omega^{<\omega}$, we can without loss of generality focus on a further restriction, namely, to the statement ``for every open set $A \subseteq \baire$, there is a Laver tree $T'$ with $\stem{T'}=\langle \rangle$ such that either $[T'] \subseteq A$ or $[T'] \cap A = \emptyset$''. This is obtained by restricting to the case $T= \omega^{<\omega}$. Accordingly, in the rest of the paper, our Laver trees are always assumed to have empty stem.
	
	As anticipated, $\lls$ is an immediate consequence of both $\SRT{1}$ and $\bSigma{0}{1}$-$\mathsf{DET}$.
	
	\begin{lemma}\label{lem:fromGP}
		Over $\rca$, $\SRT{1}$ implies $\lls$ (similarly, $\DRT{1}$ implies $\lld$).
	\end{lemma}
	The proof consists of a couple of lines, provided one identifies the Baire space with the set of infinite subsets of $\N$ (the setting for the Galvin-Prikry theorem). This is possible by Lemma \ref{lem:ramseyisbaire} below.
	\begin{proof}
		Let $A \subseteq [\N]^{<\N}$ be a code for an open set. By $\SRT{1}$, there is some infinite $X \subseteq \N$ such that $[X]^{\N} \subseteq A$ or $[X]^{\N} \cap A = \emptyset$. The tree $[X]^{<\N}$ is easily seen to be a Laver tree, so the claim follows.
	\end{proof}
	
	A similarly direct argument also shows:
	
	\begin{lemma}\label{lem:fromDET}
		Over $\rca$, $\bSigma{0}{1}$-$\mathsf{DET}$ implies $\lls$.
	\end{lemma}
	
	We skip the formal proof as it requires a little bit of non-illuminating coding. The idea is very simple though: the tree of partial plays according to a strategy for (say) Player I, is a tree which has infinite splitting at every other level. By coding two moves of the game into one, this can be seen as a Laver tree, and the determinacy of open games implies that there is one such Laver tree either fully contained or disjoint from each open set.
	
	Intuitively speaking, this shows that both $\SRT{1}$ and $\bSigma{0}{1}$-$\mathsf{DET}$ can be seen as saying that, given any open set $A$, there is a particular kind of Laver tree $T$ such that either $[T] \subseteq A$ or $[T] \cap A = \emptyset$. It is natural to ask whether the statement $\lls$, which puts no restrictions on the Laver tree, has similar proof and computability theoretic strengths. Since both $\SRT{1}$ and $\bSigma{0}{1}$-$\mathsf{DET}$ are equivalent to $\atr$ over $\rca$, one wonders whether $\lls$ is already equivalent to $\atr$. A positive answer would imply that, in some sense, the strength of the two statements above is already present in the ``bare bones'' statement about Laver trees.
	
	For now, Lemmas \ref{lem:fromGP} and \ref{lem:fromDET} show that $\lls$ is provable in $\atr$. There is another, more direct proof of $\lls$. This is an example of what Simpson calls the pseudohierarchy method, and it may be considered the ``most natural'' proof of $\lls$. Before reporting it, we recall another notion which we will use in this paper.
	
	\begin{definition}
		A \emph{Hechler tree} is a Laver tree $T \subseteq \omega^{<\omega}$ such that for every $\rho \in T$, $\concat{\rho}{n}$ is in $T$ for all but finitely many $n \in \omega$.
	\end{definition}
	Hechler trees, with possibly nonempty stems, are conditions for (a version of) Hechler forcing first considered by \cite{BL11}. We note that, unlike the property ``being a Laver tree'', the property ``being a Hechler tree'' is not a property of trees as partial orders. Indeed, any two Laver trees are isomorphic as partial orders, but not all Laver trees are Hechler trees.
	
	 We introduce some principles related to Laver and Hechler trees:
	
	\begin{definition}
		The principle $\hls$ states: ``for every open set $A \subseteq \baire$, there is either a Hechler tree $T$ such that $[T] \subseteq A$, or a Laver tree $S$ such that $[S] \cap A = \emptyset$''. The principle $\hlp$ states: ``for every open set $A \subseteq \baire$, there is either a Hechler tree $T$ such that $[T] \cap A = \emptyset$, or a Laver tree $S$ such that $[S] \subseteq A$''. The principle $\hld$ is defined analogously.
	\end{definition}
	
	Since every Hechler tree is also a Laver tree, it is immediate that either one of $\hls$ and $\hlp$ implies $\lls$ and that $\hld$ implies $\lld$. As mentioned above, a pseudohierarchy argument (see Theorems \ref{thm:hl} and \ref{thm:hlinatr} below) shows:
	
	\begin{theorem}
		$\atr$ proves both $\hls$ and $\hlp$.
	\end{theorem}

	Note that $\hls$ and $\hlp$ (in general, any Hechler-Laver principle), unlike $\lls$ and $\lld$, are true dichotomies.
	
	\begin{lemma}
		Let $T$ be a Hechler tree and $S$ a Laver tree. There is some $f \in [T] \cap [S]$.
	\end{lemma}
	\begin{proof}
		Starting from a node $\sigma \in S \cap T$, since the set $A_{\sigma}=\{n : \concat{\sigma}{n} \in T\}$ is cofinite and $B_{\sigma}=\{n : \concat{\sigma}{n} \in S\}$ is infinite, we can just take the least member of $A_{\sigma} \cap B_{\sigma}$ and proceed this way indefinitely. This yields a branch in $[S] \cap [T]$.
	\end{proof}
	
	Consequently, if $A \subseteq \baire$ is any set, and there is a Hechler tree $T$ with $[T] \subseteq A$, there cannot be a Laver tree $S$ with $[S] \cap A = \emptyset$.
	
	\subsection{Main results.}
	
	We now give a summary of our main results. On the side of reverse mathematics, the main character of the paper, namely $\lls$, resists a full characterization.
	
	\begin{theorem}
		Over $\rca$, $\lld$ implies $\aca$. Moreover, over $\rca + \bSigma{1}{1}$-$\mathsf{IND}$, $\lld$ implies $\bDelta{1}{1}$-$\mathsf{CA}_0$ and $\lls$ implies $\bSigma{1}{1}$-$\mathsf{AC}_0$. 
	\end{theorem}
	
	On the other hand, the following result suggests that the principle $\lls$ should be at the level of $\atr$.
	
	\begin{theorem}
		Over $\rca$, $\atr$ is equivalent to $\hld$ (or $\hls$, $\hlp$).
	\end{theorem} 
	
	From the computational perspective we have a much clearer picture. First, we can effectively find upper bounds (in the hyperarithmetical hierarchy) for the complexity of solutions to $\lld(A)$ (or $\hld(A)$) whose bodies are entirely contained in the clopen set coded by $A$.
	
	\begin{theorem}
	Let $C \subseteq \baire$ be a clopen set, let $A \subseteq \omega^{<\omega}$ be a prefix-free open code for $C$ and let $\alpha < \omega^A_1$ be the height of the $A$-computable wellfounded tree associated to $A$ as in Remark \ref{rem:clopencodes}. Suppose there is a Hechler (resp.\ Laver) tree $T$ with $[T] \subseteq C$. Then there is a Hechler (resp.\ Laver) tree $T'$ with $T' \leq_{\T} A^{(\alpha)}$ and $[T'] \subseteq C$.
	\end{theorem}
	
	For the case of Laver trees, the previous results can be extended: we have an effective, hyperarithmetic upper bound on the complexity of solutions to $\lls(A)$ whose bodies are contained in the open set $A$.
	
	\begin{theorem}
		There is a partial computable function $f \subc \baire \rightarrow \omega_1$ such that, if $x$ codes an open set $A$ for which there is a Laver tree $T$ with $[T] \subseteq A$, then $x \in \dom(f)$ and there is a Laver tree $T' \leq_{\T} x^{(f(x))}$ with $[T'] \subseteq A$.
	\end{theorem}
	
	We also have a detailed picture of the (arithmetical) Weihrauch degrees of the principles related to $\lls$. In Section \ref{sec:Wd} we introduce, in analogy with \cite{marconevalenti}, the relevant functions. There is a strong similarity between their degrees and those related to the open Ramsey theorem (see \cite{marconevalenti}), with a notable exception: the degree of $\mathsf{FindHS}_{\bSigma{0}{1}}$ in \cite{marconevalenti} turns out to be stronger than anything else, while our $\linds$ is at the bottom of the picture. Figure \ref{figure} presents our results and bears a striking resemblance with \cite[Figure 1]{marconevalenti}.
	
	\begin{figure}
		\begin{tikzpicture}[
			>=Stealth,
			box/.style={
				draw,
				rectangle,
				minimum width=2.2cm,
				minimum height=7mm,
				inner sep=2pt,
				fill=gray!25
			},
			strict/.style={->, line width=.3pt},
			weak/.style={->, dashed, line width=.3pt}
			]
			
			
			\draw (-6.4,-5.2) rectangle (6.4,-3.8);   
			\draw (-6.4,-3.8) rectangle (6.4,-1);   
			\draw (-6.4,-1) rectangle (6.4, 0.8);   
			\draw (-6.4, 0.8) rectangle (6.4, 3);   
			
			
			\node[box] (B) at (-3.8,2.3) {$\stCCantor \star \lls$};
			\node[box] (C) at (0,1.9) {$\CCantor \star \lls$};
			\node[box] (D) at (3.8, 1.5) {$\lls$};
			
			\node[box] (E) at (-3.8,0.1) {$\stCBaire$};
			\node[box] (F) at (0,-0.3) {$\tCBaire$};
			
			\node[box] (G) at (0,-1.9) {$\lindp \equiv_{\sW} \CBaire \equiv_{\W} \CCantor \star \wlindp$};
			
			\node[box] (H) at (0,-3.1) {$\wlindp$};
			
			\node[box] (I) at (0,-4.5) {$\UCBaire \equiv_{\sW} \wlindd \equiv_{\sW} \wlinds \equiv_{\sW} \lld \equiv_{\sW} \lindd \equiv_{\sW} \linds$};
			
			
			\draw[weak] (H) -- (G);
			
			\draw[weak] (D) -- (C);
			\draw[weak] (C) -- (B);
			
			
			\draw[strict] (I) -- (H);
			\draw[strict] (G) -- (F);
			\draw[strict] (F) -- (C);
			\draw[strict] (E) -- (B);
			
			\draw[strict] (F.west) -- (E.east);
			
			\draw[strict]
			(H.east)
			.. controls +(3.4,0) and +(0,-1.6) ..
			(D.south);
			
		\end{tikzpicture}
		\caption{This table summarizes the relative position of the (arithmetical) Weihrauch degrees of functions related to the open and clopen Laver partition theorem. The large rectangles correspond to arithmetical Weihrauch degrees. These are linearly ordered by $\leq^{\ari}_{\W}$ from the bottom to the top of the table. We draw a dashed arrow from $A$ to $B$ to indicate that $A \leq_{\W} B$ (but we do not know whether $B \leq_{\W}A$). A solid arrow from $A$ to $B$ indicates that $A <_{\W} B$.}\label{figure}
	\end{figure}
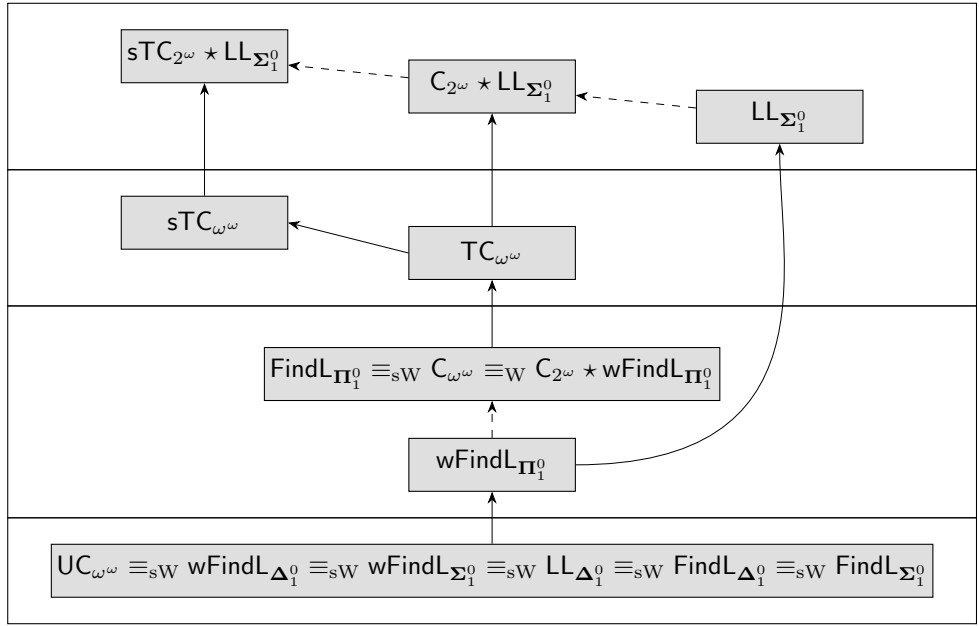

	\section{Preliminaries}
	
	We adopt the convention to use the symbol $\N$ only when referring to the first order part of the universe of a structure in the language of second order arithmetic. In particular, we always use $\N$ in the statements and proofs of our reverse mathematical results. Accordingly, we always use the symbol $\omega$ to denote the actual natural numbers. Now we fix some notation related to finite and infinite strings of natural numbers.
    
    We write $\omega^{<\omega}$ and $2^{<\omega}$ to denote the sets of, respectively, finite strings of natural numbers, and finite binary strings. The Baire space and Cantor space are denoted as $\baire$ and $\cantor$. Given two finite strings $\sigma$ and $\tau$, we denote their concatenation as $\concat{\sigma}{\tau}$. In case $\tau=\langle n \rangle$ for some natural number $n$, we write $\concat{\sigma}{n}$ instead of $\concat{\sigma}{\langle n\rangle}$. Given $\sigma \in \omega^{<\omega}$ and $\tau \in \omega^{<\omega} \cup \baire$, we write $\sigma \preceq \tau$ if $\sigma$ is a prefix of $\tau$. A set $A \subseteq \omega^{<\omega}$ is called \emph{prefix-free} if for every $\sigma, \tau \in A$, $\sigma \npreceq \tau$. A tree is a set $T \subseteq \omega^{<\omega}$ which is closed under prefixes, i.e.\ if $\sigma \in T$ and $\tau \preceq \sigma$, then $\tau \in T$. The \emph{body} of a tree $T$, denoted as $[T]$, is the set $\{f \in \baire : \forall n \,\, (f \upto n \in T)\}$. A tree $T$ is \emph{wellfounded} if $[T]=\emptyset$, otherwise it is \emph{illfounded}. We denote the \emph{height} (tree rank) of a wellfounded tree $T$ as $\height(T)$. We fix a (suitably effective) standard bijection $i \colon \omega \rightarrow \omega^{<\omega}$.
	For definiteness, we use the function $i$ explicitly in these preliminary definitions. Throughout the paper we will silently identify natural numbers with finite strings of natural numbers, as is customary in computability.
    
    We assume that the reader is familiar with the basic notions of computability theory and the type-2 theory of effectivity (i.e.\ computability on $\baire$ and $\cantor$, and the machinery needed to endow so-called \emph{represented spaces} with a notion of computability deriving from that on $\baire$). Some standard references for these theories are, respectively, \cite{rogers} and \cite{Brattka2021}. We also assume familiarity with the basic notions of topology and descriptive set theory on $\baire$ and $\cantor$. As mentioned in the introduction, we use some computational properties of the Galvin-Prikry theorem to derive properties of the Laver partition theorem, so we make reference to the \emph{Ramsey space} $\ramsey$ consisting of the infinite subsets of the natural numbers. We also make use of tools from hyperarithmetical theory and common results from reverse mathematics. The standard references for these areas are, respectively, \cite{sacks2017} and \cite{simpson}.
	
	Here we briefly recall some bits about the topics above, to set the stage and fix the notation we use throughout the paper.

    \subsection{Some subsystems of second order arithmetic}
	
	We recall the basics of reverse mathematics and introduce some theories which are relevant for this paper. Reverse mathematics is concerned with studying theories in the language of second order arithmetic, with the aim of pinning down the axioms necessary to prove theorems of ordinary mathematics.
	
	The standard base theory is $\rca$, consisting of the basic arithmetic axioms together with the axiom schema of $\bDelta{0}{1}$-comprehension and $\bSigma{0}{1}$-induction.
	Other theories featured in this paper are (in increasing order of strength) $\wkl$, consisting of $\rca$ together with the axiom stating that every infinite subtree of $2^{<\N}$ has a branch; $\aca$, consisting of $\rca$ together with the axiom schema of arithmetical comprehension; $\bDelta{1}{1}$-$\mathsf{CA}_0$, consisting of $\rca$ together with the axiom schema of $\bDelta{1}{1}$-comprehension; $\bSigma{1}{1}$-$\mathsf{AC}_0$, consisting of $\rca$ together with the axiom schema of $\bSigma{1}{1}$-choice, namely the axiom
	\[(\forall i \,\, \exists f \,\, \varphi(f,i)) \rightarrow (\exists f \,\, \forall i \,\, \varphi((f)_i, i))\]
	for all $\bSigma{1}{1}$ formulae $\varphi$; and lastly $\atr$, consisting of $\rca$ together with the schema of arithmetical transfinite recursion. We now recall some basic tools to work with these theories.
	
	\begin{lemma}[{\cite[Lemma III.1.3]{simpson}}]
		Over $\rca$, the theory $\aca$ is equivalent to the statement: ``for every injective function $f \colon \N \rightarrow \N$, there is a set $R$ such that $\forall n \,\, (n \in R \leftrightarrow \exists i \,\, f(i)=n)$''.
	\end{lemma}
	
	Over $\aca$, we have more manageable equivalent formulations of the theories $\bDelta{1}{1}$-$\mathsf{CA}_0$ and $\bSigma{1}{1}$-$\mathsf{AC}_0$. This is due to the following fact. 
	
	\begin{lemma}{\cite[Lemma IV.4.4]{simpson}}\label{lem:knf}
		$\aca$ proves Kleene's normal form theorem for $\bSigma{1}{1}$ predicates, i.e.\ the statement that there is a computable procedure mapping each $\bSigma{1}{1}$ formula $\varphi$ to a tree $T_{\varphi}$ such that $\varphi \leftrightarrow \exists f \,\, (f \in [T_{\varphi}])$.
	\end{lemma}
	
	We stress that in the lemma above the map $\varphi \mapsto T_{\varphi}$ is not only arithmetically definable but computable. The strength of $\aca$ is only needed to show the equivalence between the formulae $\varphi$ and $\exists f \,\, (f \in [T_{\varphi}])$.
	 
	The next two characterizations are obtained by a straightforward application of Kleene's normal form theorem.
	
	\begin{proposition}\label{prop:doowithtrees}
		Over $\aca$, $\bDelta{1}{1}$-$\mathsf{CA}_0$ is equivalent to the statement ``for every double sequence of trees $(T^0_i, T^1_i)_{i \in \N}$ such that for every $i$ exactly one of $T^0_i$ and $T^1_i$ is wellfounded, there is a set $Z$ such that $i \in Z \leftrightarrow [T^0_i] \neq \emptyset$''.
	\end{proposition}
	
	\begin{proposition}\label{prop:soowithtrees}
		Over $\aca$, $\bSigma{1}{1}$-$\mathsf{AC}_0$ is equivalent to the statement ``for every sequence of illfounded trees $(T_i)_{i \in \N}$ there is a sequence $(g_i)_{i \in \N}$ such that $g_i \in [T_i]$ for every $i$''.
	\end{proposition}
	
	The theory $\atr$ has a similar formulation in terms of sequences of trees, also known as $\bSigma{1}{1}$-separation.
	
	\begin{proposition}\label{prop:atrwithtrees}
		Over $\aca$, $\atr$ is equivalent to the statement ``for every double sequence of trees $(T^0_i, T^1_i)_{i \in \N}$ such that for every $i$ at most one of $T^0_i$ and $T^1_i$ is wellfounded, there is a set $Z$ such that $\forall i \,\, (i \in Z \leftrightarrow [T^0_i] \neq \emptyset)$''.
	\end{proposition}
	\begin{proof}
		See \cite[Theorem V.5.1]{simpson} for the version of this result with $\bSigma{1}{1}$ formulae instead of trees. The version with trees follows from Kleene's normal form theorem.
	\end{proof}
	We briefly go over induction schema. Given a class of formulas $\Gamma$, the schema $\Gamma$-$\mathsf{IND}$ (or $\mathsf{I}\Gamma$) is the set of axioms
	\[\varphi(0) \land \forall n \, (\varphi(n) \rightarrow \varphi(n+1)) \rightarrow \forall n \, \varphi(n)\]
	where $\varphi$ ranges over the formulas in $\Gamma$.
	
	It is well known that $\Gamma$-comprehension implies $\mathsf{I}\Gamma$, which in turn is equivalent to bounded $\Gamma$-comprehension, namely the fact that, if $\varphi \in \Gamma$ and $n \in \N$, the set $\{m < n : \varphi(m)\}$ exists. An axiom schema weaker than $\mathsf{I}\Gamma$ is the schema of \emph{bounding} for $\Gamma$-formulas $\mathsf{B}\Gamma$, which is the set of axioms
    \[
    \forall m \, (\forall i<m \, \exists n \, \varphi (i,n) \rightarrow \exists b \, \forall i<m \, \exists n<b \, \varphi (i,n))
    \]
    where $\varphi$ ranges over the formulas in $\Gamma$.
    We mention the following:
	
	\begin{proposition}\cite[Proposition 6.5.4]{DM}\label{prop:fuf}
		Over $\rca$, the following are equivalent:
		\begin{itemize}
			\item $\mathsf{B}\bSigma{0}{2}$,
			\item the statement ``the union of finitely many finite sets is finite''.
		\end{itemize}
	\end{proposition}
	
	We conclude the section introducing \emph{finitary $\bSigma{1}{1}$-$\mathsf{AC}_0$},\footnote{Beware that in \cite{Goh_finite_choice} Goh studied a principle called finite $\bSigma{1}{1}$-$\mathsf{AC}_0$, which is quite different from ours.} namely the schema 
	\[\forall i \,\, \exists f \,\, \varphi(i,f) \rightarrow \forall n \,\, \exists g\,\, \forall i \leq n \,\, \varphi(i, g_i)\]
	where $\varphi$ is any $\bSigma{1}{1}$ formula.
	Note that if $T=(T_i)_{i \in \N}$ is a sequence of illfounded trees, $\rca$ + finitary $\bSigma{1}{1}$-$\mathsf{AC}_0$ proves that for all $n \in \N$, there is a sequence $(g_i)_{i \leq n}$ with $g_i \in [T_i]$ for every $i \leq n$. 
	
	\begin{lemma}\label{lem:finitechoice}
		Over $\rca$, both $\bSigma{1}{1}$-$\mathsf{AC}_0$ and $\bSigma{1}{1}\text{-}\mathsf{IND}$ imply finitary $\bSigma{1}{1}$-$\mathsf{AC}_0$.
	\end{lemma}
	\begin{proof}
		The fact that $\bSigma{1}{1}$-$\mathsf{AC}_0$ implies finitary $\bSigma{1}{1}$-$\mathsf{AC}_0$ is obvious. 
		
		To see that $\bSigma{1}{1}\text{-}\mathsf{IND}$ also implies finitary $\bSigma{1}{1}$-$\mathsf{AC}_0$, it suffices to apply $\bSigma{1}{1}\text{-}\mathsf{IND}$ to the formula $\psi(n)$ defined by $\exists g \,\, \forall i \leq n \,\,  \varphi(g_i)$.
	\end{proof}
	Since $\bSigma{1}{1}$-$\mathsf{AC}_0$ and $\bSigma{1}{1}\text{-}\mathsf{IND}$ do not imply each other, our new principle cannot be equivalent to either. In Section \ref{sec:rm}, we will obtain some partial reversals of the principles of our interest over the base theory $\rca$ + finitary $\bSigma{1}{1}$-$\mathsf{AC}_0$.
    
	\subsection{Preliminaries on Weihrauch degrees}
	
	\subsubsection{Basic definitions}
	Weihrauch reducibility is a concept which allows us to classify relations in terms of their computational strength.
	
	\begin{definition}
		A \emph{representation} of a set $X$ is a surjective, possibly partial, function $\delta \subc \baire \rightarrow X$. The pair $(X,\delta)$ is called a represented space.
	\end{definition}
	
	Represented spaces inherit a notion of computability from the usual one on $\baire$. This notion is then used to define the preorder $\lew$ of Weihrauch reducibility on relations defined on represented spaces. A relation $R \subseteq X \times Y$, where $X$ and $Y$ are represented spaces, is usually thought of as a (possibly partial) multi-valued function. The common notation, accordingly, is of the form $r \colon X \multif Y$ (resp.\ $r \subc X \multif Y$). If $r$ is as above, and $x \in \dom(r)$, then $r(x)=\{y \in Y : R(x,y)\}$. We often refer to partial multi-valued functions as just ``functions'' or ``problems'', when our intended meaning is clear from context.
	
	\begin{definition}
		Let $f \subc (X, \delta_X) \multif (Y, \delta_Y)$ be a partial multi-valued function between represented spaces. A \emph{realizer} for $f$ (in symbols $F \vdash f$) is a function $F \subc \baire \rightarrow \baire$ such that
		\begin{itemize}
			\item $\dom(f \circ \delta_X) \subseteq \dom(F)$, and 
			\item for every $x \in \dom(f \circ \delta_X)$, $\delta_Y(F(x)) \in f(x)$.
		\end{itemize}
	\end{definition}
	
	We can now give the definition of Weihrauch reduction.
	
	\begin{definition}
		Let $f \subc (X, \delta_X) \multif (Y, \delta_Y)$ and $g \subc (A, \delta_A) \multif (B, \delta_B)$ be partial multi-valued functions between represented spaces. We say that $f$ is Weihrauch reducible to $g$, in symbols $f \lew g$ if there are computable functions on Baire space $\varphi$ and $\psi$ such that, for every $G \vdash g$, the function $p \mapsto \psi(\langle p, G(\varphi(p))\rangle)$ is a realizer for $f$. The reduction is \emph{strong}, in symbols $f \lews g$, if $p \mapsto \psi( G(\varphi(p)))$ is a realizer for $f$, i.e.\ if the postprocessing $\psi$ does not need the original input. If $\varphi$ and $\psi$ are not computable but merely arithmetical, then we say that $f$ arithmetically Weihrauch reduces to $g$, in symbols $f \leq^{\ari}_{\W} g$.
	\end{definition}
	
	The structure of Weihrauch degrees consists of the quotient of the preorder $\lew$ defined above under the relation of bi-reducibility. Similarly one obtains the structure of the strong and arithmetical Weihrauch degrees.
	
	Note that if $f \subc X \multif Y$ is any partial function, then $f \equiv_{\sW} \bigcup_{F \vdash f}F$, hence each (strong, arithmetical) Weihrauch degree has a representative which is a relation on Baire space. We implicitly use this throughout the paper, where we work with the codes of the objects of our interest.
	
	\begin{remark}
		Any theorem of the form $\forall x \in A \,\, ( \psi(x) \rightarrow \exists y \in B\,\, \varphi(x,y))$, where $A$ and $B$ are represented spaces, can be viewed as the partial multi-valued function $t \subc A \multif B$ with $\dom(t) = \{a \in A : \psi(a)\}$ and $t(a)=\{b \in B : \varphi(a,b)\}$. Studying the Weihrauch degree of $t$ gives information on the computational content of the original theorem. This approach to the study of the logical strength of theorems, which started with \cite{GherardiMarcone}, complements that of reverse mathematics.
	\end{remark}
	
	\subsubsection{Common spaces and their representations}\label{sub:coding}
	
	\begin{definition}
		Let $A \subseteq \omega^{<\omega}$. We say that $A$ \emph{codes the open set} $\open{A}=\{f \in \baire : \exists \tau \prec f \,\, (\tau \in A)\}$.
		The function $\open{\cdot} \colon \baire \rightarrow \bSigma{0}{1}(\baire)$ defined as $f \mapsto \open{\{i(k) : \exists n \in \omega\, k+1 = f(n)\}}$ is a representation of the set of open subset of $\baire$. 
	\end{definition}
	
	We point out the abuse of notation: we use symbols $\open{A}$ and $\open{f}$ to denote the open sets coded by, respectively, a set of strings $A$, and a function $f \in \baire$. The simple idea is that, for any $A$, $\open{A}=\open{f_A}$ where $f$ enumerates the (successors of codes for) the strings in $A$. We trust that this creates no confusion.
	
	We also point out that, without loss of generality, we can assume that our codes for open sets are prefix-free: there is a uniformly computable procedure which, given any $A \subseteq \omega^{<\omega}$, outputs a set $\pf{A}$ which is prefix-free and has the property that any string $\sigma \in A$ has some prefix $\tau \preceq \sigma$ such that $\tau \in \pf{A}$.
	
	From the representation $\open{\cdot}$ we derive representations for closed and clopen subsets of $\baire$.
	
	\begin{definition}
        We define a representation of closed sets $\delta_{\bPi{0}{1}} \colon \baire \rightarrow \bPi{0}{1}(\baire)$ by mapping $f \in \baire$ to $\baire \setminus \open{f}$. 		
		We also define a representation of clopen sets as follows: a code for a $\bDelta{0}{1}$ set $C$ is a pair $(f,g)$ with $C=\open{f}=\baire \setminus \open{g}$.
	\end{definition}
	
	\begin{remark}\label{rem:clopencodes}
        We note that there is at least one more useful representation of closed sets. Given any $A \subseteq \omega^{<\omega}$, we define a tree $T_A$ as $T_A=\{\sigma \in \omega^{<\omega} : \nexists \tau \,\, (\tau \in \pf{A} \land \tau \preceq \sigma)\}$.
		It is easy to see that $[T_A]= \baire \setminus \open{A}$. 
        In the other direction, given a tree $T$, we notice that $[T]= \baire \setminus \open{\omega^{<\omega} \setminus T}$. This justifies thinking about codes for closed sets as trees.
        
		Moreover, if $C$ is a clopen subset of $\baire$ and $A \subseteq \omega^{<\omega}$ is a code for $C$ as an open set, then it is easy to see that closing $\pf{A}$ under prefixes yields a wellfounded tree. 
		Indeed, this is a characterization: an open set is clopen iff any open code for it satisfies the above property.
	\end{remark}
	
	\begin{definition}
		The \emph{Ramsey space} $\ramsey$ is the set of infinite subsets of $\omega$. Given an infinite $A \subseteq \omega$, we think of $A$ in terms of its \emph{principal function} $p_A$, i.e.\ its strictly increasing enumeration. With this identification, we get $\ramsey \subseteq \baire$. 
	\end{definition}
	
	\begin{lemma}[Folklore]\label{lem:ramseyisbaire}
		There is a computable homeomorphism between $\ramsey$ (endowed with the subspace topology of $\baire$) and $\baire$.
	\end{lemma}

    \begin{theorem}[{Galvin-Prikry, see \cite[Theorem 19.11]{kechris}}]
        Every Borel subset $A$ of $\ramsey$ is Ramsey, i.e.\ there is some $X \in \ramsey$ such that either $[X]^{\omega} \subseteq A$ or $[X]^{\omega} \cap A = \emptyset$.
    \end{theorem}

    When we use the Galvin-Prikry theorem to prove facts about the Laver Partition Theorem we tacitly identify $\baire$ with $\ramsey$: this is justified in light of Lemma \ref{lem:ramseyisbaire}.
	
	We now introduce the natural representation of Laver trees.
	This kind of representation is implicit in the fusion arguments which are common in applications of Laver forcing.
	
	\begin{definition}\label{def:codesl}
		We simultaneously define the set of partial codes $C \subseteq \omega^{<\omega}$ for Laver trees and a decoding function $T \colon C \rightarrow [\omega^{<\omega}]^{<\omega}$, as follows:
		\begin{itemize}
			\item $\langle \rangle \in C$, and $T(\langle \rangle)= \emptyset$.
			\item if $\sigma \in C$ and $T(\sigma)$ has already been defined, then $\concat{\sigma}{n} \in C$ if and only if (i) $n$ codes a set of strings $D \subseteq \omega^{<\omega}$ of cardinality $|T(\sigma)|$ such that every string in $T(\sigma)$ has exactly one immediate successor in $D$ and (ii) if $\{\tau, \concat{\tau}{i}\} \subseteq T(\sigma)$ and $\concat{\tau}{j} \in D$, then $j>i$. If these conditions are met, we let $T(\concat{\sigma}{n})=T(\sigma)\cup D$.
		\end{itemize}
	\end{definition}
	
	It is easy to see that (i) $C$ is a computable Laver tree, (ii) for every $\sigma \in C$, $T(\sigma)$ is a finite tree, and (iii) if $g \in [C]$, then $\bigcup_{n \in \omega}T(g \upto n)$ is a Laver tree.
	
	Since $C$ is a computable Laver tree, there is a straightforward computable isomorphism of trees $j \colon \omega^{<\omega} \rightarrow C$ which preserves the lexicographic order. The function $j$ induces the computable homeomorphism $J \colon \baire \rightarrow [C]$.
	
	\begin{definition}
		We define the representation $\LL$ of the set of Laver trees as $\LL(f)=\bigcup_{n \in \omega}T(J(f) \upto n)$ for every $f \in \baire$.
	\end{definition}
	
	One can see that the representation $\LL$ is actually a bijection, and it is computably equivalent to the representation of Laver trees via characteristic functions. We adopt the representation $\LL$ for convenience (for example, using $\LL$ allows us to see every element of $\baire$ as a code for a Laver tree).
	
	\subsubsection{Some benchmarks and some operations on the Weihrauch degrees}
	
	We now introduce several well-known benchmark functions which are useful to characterize the Weihrauch degrees associated to theorems of proof theoretic strength around that of $\atr$. We also introduce some standard operations to make new Weihrauch degrees from old ones. These often have an immediate computational meaning, and they are widely used to establish reductions and separations.
	
	\begin{definition}
		Let $\tr$ be the space of trees on $\omega$, represented via characteristic functions.
		We define the following functions:
		\begin{itemize}
			\item $\mathsf{LPO} \colon \cantor \rightarrow 2$ defined as $\mathsf{LPO}(p)=1$ if there is some $n$ such that $p(n)=1$ and $\mathsf{LPO}(p)=0$ otherwise,
			\item $\mathsf{WF} \colon \tr \rightarrow 2$ defined as $\mathsf{WF}(T)=1$ if $[T]\neq \emptyset$ and $\mathsf{WF}(T)=0$ otherwise,
			\item $\lim \subc \baire \rightarrow \baire$, mapping a convergent sequence (coded as a single real) to its limit,
			\item $\UCBaire \subc \tr \rightarrow \baire$, with $\dom(\UCBaire)=\{T \in \tr : |[T]|=1 \}$ and, $\UCBaire(T)=x$ if and only if $[T]=\{x\}$,
			\item $\CBaire \subc \tr \multif \baire$, with $\dom(\CBaire)=\{T \in \tr : [T] \neq \emptyset \}$, defined as $\CBaire(T)=[T]$,
			\item $\CCantor \subc \tr \multif \cantor$, with $\dom(\CCantor)=\{T \in \tr : T \subseteq 2^{<\omega} \land [T]\neq \emptyset\}$, defined as $\CCantor(T)=[T]$
		\end{itemize}

        The function $\UCBaire$ is known as unique choice on Baire space, while $\CBaire$ and $\CCantor$ are known as choice on Baire and Cantor space, respectively.
      \end{definition}

    We define operations on multi-valued functions between represented spaces.

	\begin{definition}
        We define the totalization and strong totalization operators. They are commonly applied to choice problems. Let $f \subc X \multif Y$ be any problem:
		\begin{itemize}
			\item 
            $\mathsf{t}f \colon X \multif Y$
            is defined by $\mathsf{t}f(x)=f(x)$ if $x \in \dom(f)$, and $\mathsf{t}f(x)=Y$ otherwise,
			\item $\mathsf{st}f \colon X \multif 2 \times Y$ is defined by $\mathsf{st}f(x)=\{1\}\times f(x)$ if $x \in \dom(f)$, and $\mathsf{st}f(x)=\{0\} \times Y$ otherwise.
		\end{itemize}
	\end{definition}
	\begin{definition}
		Let $f \subc X \multif Y$ and $g \subc A \multif B$ be partial multi-valued functions. We define:
		\begin{itemize}
			\item the parallel product of $f$ and $g$, as $(f \times g) \subc X \times A \rightarrow Y \times B$, with $(f\times g)(x,a)=f(x) \times g(a)$,
			\item the parallelization of $f$, as $\widehat{f} \subc X^{\omega} \rightarrow Y^{\omega}$ with $\widehat{f}((p_i)_{i \in \omega})=\prod_{i \in \omega}f(p_i)$,
		\end{itemize}
		We also use the compositional product $f \star g$. The precise definition is a bit technical, so we skip it and refer the interested reader to \cite[Definition 11.5.3]{Brattka2021}. For our purposes, $f \star g$ corresponds to $g$ followed by a computable operation and then $f$.
	\end{definition}
	
	We conclude our definitions with that of cylinder, a type of function which, intuitively, remembers its input.
	
	\begin{definition}
		A function $f \subc X \multif Y$ is a cylinder if $\id{{\baire}} \times f \equiv_{\sW} f$.
	\end{definition}
	
	We end the section with two Proposition summarizing some well-known facts about the benchmarks introduced above which we will use extensively in the paper.
	\begin{proposition}\label{prop:summary}
		The following hold.
		\begin{enumerate}
			\item Being a cylinder is preserved under strong Weihrauch equivalence,
			\item the functions $\CCantor$, $\lim$, $\CBaire$ and $\UCBaire$ are all cylinders,
			\item $\CBaire$ is closed under composition, i.e.\ if $f \lew \CBaire$ and $g \lew \CBaire$, then $f \star g \lew \CBaire$,
			\item $\tCCantor \equiv_{\sW} \CCantor$,
			\item $\lim \equiv_{\sW} \J$, where $\J$ denotes the Turing jump operator,
			\item for every $f$ and $g$, we have $f \leq^{\ari}_{\W} g$ if and only if there are $n$ and $m$ such that $f \lew \lim^{[n]} \star g \star \lim^{[m]}$.
		\end{enumerate}
	\end{proposition}

    It will be convenient to consider several incarnations of the function $\UCBaire$. Consider the following functions:
    \begin{itemize}
    \item $\atrw \colon \wo \times \omega \times 2^{\omega} \rightarrow 2^{\omega}$,\footnote{Here $\wo$ denotes the set of wellorders on $\omega$, represented via characteristic functions.} maps a wellorder $\alpha$, a code $n$ for an arithmetical operator $\varphi$, and a set $X \in 2^{\omega}$ to the (unique) hierarchy obtained iterating the operator $\varphi$ along $\alpha$, starting from $X$,
    \item $\bDelta{1}{1}$-$\mathsf{CA}$ maps a $\bDelta{1}{1}$ description of a set of natural numbers (i.e.\ a double sequence of trees $(T^0_i, T^1_i)_{i \in \omega}$ such that for every $i$ exactly one of $T^0_i$ and $T^1_i$ is wellfounded) to the characteristic function of the set it defines.
    \end{itemize}
    \begin{proposition}\label{prop:ucbequiv}
        $\UCBaire \equiv_{\sW} \atrw \equiv_{\sW} \bDelta{1}{1}$-$\mathsf{CA}$.
    \end{proposition}
    \begin{proof}
        See \cite[Theorems 3.11 and 3.13]{kmp}.
	\end{proof}
	
	\subsection{An explicitly quantitative version of Kleene's hyperarithmetical quantifier theorem}

	We recall that the Spector-Gandy theorem states that $\lPi{1}{1}$ predicates coincide with predicates of the form $\exists f \in \lDelta{1}{1} \,\, \varphi(f)$, where $\varphi$ is a $\lPi{0}{1}$ formula.
	This theorem is uniformly effective, namely, there are computable functions transforming each $\lPi{1}{1}$ formula $\varphi$ into an equivalent formula of the form $\exists f \in \lDelta{1}{1} \,\, \psi(f)$ with $\psi$ of complexity $\lPi{0}{1}$, and vice versa. In the converse direction, one can even take $\psi$ to be $\lPi{1}{1}$ and uniformly and effectively obtain a $\lPi{1}{1}$ formula $\varphi$ which is equivalent to $\exists f \in \lDelta{1}{1} \,\, \psi(f)$.
	The statement that predicates of the form $\exists f \in \lDelta{1}{1} \,\, \psi(f)$, with $\psi \in \lPi{1}{1}$, are $\lPi{1}{1}$, is known as Kleene's hyperarithmetical quantifier theorem.
	
	Combining Kleene's hyperarithmetical quantifier theorem with Kleene's normal form theorem for $\lPi{1}{1}$ predicates we obtain a computable function $T$ mapping predicates of the form $\chi=\exists f \in \lDelta{1}{1}\,\, \psi(f)$ to trees $T(\chi)$ such that, for every $\chi$, $\chi \leftrightarrow [T(\chi)] = \emptyset$. This function is extracted essentially from the proof of the two aforementioned theorems. We show that the standard proofs yield an additional quantitative bound on the minimal complexity of a witness $f$ for $\psi$.
	
	\begin{theorem}\label{thm:qsg}
		There is a computable function $T$ mapping formulae $\chi$ of the form $\exists f \in \lDelta{1}{1} \,\, \psi(f)$ with $\psi \in \lPi{1}{1}$ to trees such that, if $\chi$ holds, then $T(\chi)$ is wellfounded and there is a witness $f$ such that $\psi(f)$ and $f \leq_{\T} 0^{(\alpha)}$ where $\alpha=\height(T(\chi))$.
	\end{theorem} 
	
	We believe this result is folklore. However, we were unable to find this quantitative bound explicitly in the literature. For this reason we repeat the usual proof found in {\cite[Corollary II.1.4 and Proposition I.5.3]{sacks2017}}, sketching a proof of the complexity bound.
	
	\begin{proof}
		Consider a predicate of the form $\chi= \exists f \in \lDelta{1}{1} \,\, \psi(f)$, where $\psi \in \lPi{1}{1}$. By the well-known characterization of hyperarithmetic sets, this is equivalent to
		\[\rho : \exists e \, \exists a \,\, (a \in \ko \land \forall X \, (\mathrm{H}(a, X) \rightarrow \Phi_e(X) \text{ is total} \land \psi(\Phi_e(X)))\]
		where $\mathrm{H}(a,X)$ is the $\lPi{0}{2}$ formula stating that $X$ is the (unique) jump hierarchy starting from $\emptyset$ along the linear order coded by $a$. It is not hard to see that $\rho$ is $\lPi{1}{1}$. Now Kleene's normal form theorem gives a computable procedure to define a tree $T$ that is wellfounded if and only if $\rho$ holds. One can define this tree from $\rho$ attacking one subformula at a time, and in this context the closure properties of $\lPi{1}{1}$ formulae correspond to suitable operation on trees.
		
		For example, the existential number quantifiers correspond to interleaving of trees, so that, for every pair $(e,a)$ one can define a tree $T_{e, a}$ which is wellfounded if and only if
		\[a \in \ko \land \forall X \, (\mathrm{H}(a, X) \rightarrow \Phi_{e}(X) \text{ is total} \land \psi(\Phi_{e}(X))).\]
		The tree corresponding to $\rho$ is the interleaving of the countably many $(T_{e,a})_{e,a \in \omega}$. Similarly, conjunction corresponds to the disjoint sum of trees. Thus for every fixed $(e, a)$, one has the tree $T^0_{e, a}$ corresponding to the formula $a \in \ko$, and the tree $T^1_{e, a}$ corresponding to the formula ``$\forall X \, (\mathrm{H}(a, X) \rightarrow \Phi_{e}(X) \text{ is total} \land \psi(\Phi_{e}(X)))$''. Note that the standard definition of $T^0_{e, a}$ is the tree of finite $<_{a}$-decreasing sequences, and if $a \in \ko$ then $\height(T^0_{e,a})=|a|$, the ordinal coded by $a$.
		
		It is known that, if $(T_n)_{n \in \omega}$ is a sequence of trees and at least one is wellfounded, then the interleaving of the trees is wellfounded, and its height is greater than or equal to that of the ``shortest'' wellfounded tree in the sequence. On the other hand, the sum of $(T_n)_{n \in \omega}$ is wellfounded if and only if all of the $T_n$'s are wellfounded, and the height of the sum majorizes that of each individual summand.
		
		Now assume $\chi$ holds, so there is some $\beta < \ock$ and some $f \leq_{\T} 0^{(\beta)}$ with $\psi(f)$. Fix some $\hat{e}$ and $\hat{a}$ such that $|\hat{a}|=\beta$ and if $X$ is such that $\mathrm{H}(a,X)$, then $\Phi_e(X)=f$. We obtain that $T_{\hat{e},\hat{a}}$, and hence $T(\chi)$, is wellfounded. Moreover, $\beta=|\hat{a}| =\height(T^0_{\hat{e}, \hat{a}})\leq \height(T_{\hat{e},\hat{a}}) \leq \height(T(\chi))$ and hence $f \leq_{\T} 0^{(\height(T(\chi)))}$.
	\end{proof}
	
	\section{Computability theoretic bounds}
	
	In this section we sketch a $\zfc$ proof of $\hls$ and $\hlp$ (for a detailed proof, see \cite{GO}), to extract computability theoretic bounds for the solutions of these problems. These bounds will help us classify the Weihrauch degrees of functions related to the Laver Partition theorem. The first proof of this result, which in full generality applies to $\bSigma{1}{1}$ sets, is in \cite{miller}, and uses similar ideas. 
	
	\begin{theorem}\label{thm:hl}
		Let $A \subseteq \omega^{<\omega}$ be the code for an open set. There is either a $\lDelta{1}{1}(A)$ Laver tree $T$ with $[T] \subseteq \open{A}$, or a Hechler tree $S$ with $[S] \subseteq \baire \setminus \open{A}$.
		
		Moreover, there is either a Hechler tree $T$ with $[T] \subseteq \open{A}$ or a Laver tree $S$ with $[S] \subseteq \baire \setminus \open{A}$.
	\end{theorem}
	\begin{proof}
		This is a ranking and overspill argument, also called pseudohierarchy argument in \cite{simpson}. Without loss of generality, assume that $A$ is prefix-free. For simplicity, we will also assume that $A$ is computable.
		
		We define a sequence of sets $(R_{\alpha})_{\alpha \in \omega_1}$ as follows: $R_0=\{\sigma : \exists \tau \preceq \sigma \,\, (\tau \in A)\}$ and, for every $\alpha >0$, $R_{\alpha}=\{\sigma : \alpha=\min\{\gamma : \einf n \,\, \exists \beta < \gamma \, (\concat{\sigma}{n} \in R_{\beta})\}\}$. Note that the hierarchy of the $R_{\alpha}$ is defined iterating a $\lPi{0}{2}$ operator. We say tha $\sigma$ has rank $\alpha$ if $\sigma \in R_{\alpha}$. An application of $\lSigma{1}{1}$-boundedness shows that $R_{\ock}=\emptyset$, i.e.\ either a string $\sigma$ is ranked with a computable ordinal, or it never gets ranked.
		
		If $\langle \rangle$ has rank $\alpha < \ock$, then the hierarchy $(R_{\beta})_{\beta \leq \alpha}$ is computable in $0^{(2 \cdot \alpha)}$, and the $(R_{\beta})_{\beta \leq \alpha}$-computable (and hence hyperarithmetic) tree $T$ of rank-descending sequences (with the convention that $0<0$, so that if $\sigma \in T$ has rank $0$, then $\concat{\sigma}{n} \in T$ for every $n$) is a Laver tree with $[T] \subseteq \open{A}$.
		
		On the other hand, if $\langle \rangle$ never gets ranked, then by overspill there must be some illfounded linear order $L$ along which one can build a hierarchy $(R_l)_{l \in L}$ as above, with $\langle \rangle \notin \bigcup_{l \in L} R_l$. Note that if $l <_L n$ and $\sigma \notin R_n$, it follows that finitely many of the successors of $\sigma$ are in $R_l$. Hence, given the hierarchy $(R_l)_{l \in L}$ and an $<_L$ descending sequence $f$, we can recursively build a tree as follows. We start at stage $0$ with $\langle \rangle$ and, at stage $s+1$, we take all the successors of the current leaves that are not $f(s)$-ranked. This yields a Hechler tree $S$ with $[S] \subseteq \baire \setminus \open{A}$.
		
		To swap the roles of the open and closed set, it suffices to consider hierarchies defined as above, with the quantifier $\finf$ replacing $\einf$. Note that in this case we cannot use a $\lSigma{1}{1}$-bounding argument to show that ranks are strictly below $\ock$.
	\end{proof}
	
	\begin{theorem}\label{thm:hlinatr}
		$\atr$ proves $\hls$ and $\hlp$.
	\end{theorem}
    \begin{proof}
         We can use arithmetical transfinite recursion to build the hierarchies of the proof of Theorem \ref{thm:hl}. The overspill argument is also available in $\atr$ (cf.\ the pseudohierarchies of \cite[Lemma V.4.12]{simpson}).
    \end{proof}
	
	\begin{remark}
		We point out an asymmetry in the construction of the proof of Theorem \ref{thm:hl}, which becomes visible working in models of $\atr$. Let $M$ be an $\omega$-model of $\atr$ which is not a $\beta$-model, i.e.\ such that there is a linear order $L \in M$ which is illfounded in $\mathcal{P}(\omega)$, but wellfounded from the point of view of $M$. In this context, it is possible that there is some code for an open set $A$ and a ranking for $A$ along $L$ which ranks the empty string. With this ranking $M$ might build a Laver tree $T$ such that $M \vDash [T] \subseteq \open{A}$, while in reality $\mathcal{P}(\omega) \vDash [T] \nsubseteq \open{A}$. This cannot happen in the closed case, as $\bSigma{1}{1}$ facts are upwards absolute. Hence, the construction of a Hechler tree in $M$ automatically yields a tree $S$ such that $\mathcal{P}(\omega) \vDash [S] \cap \open{A}= \emptyset$.
	\end{remark}
	
	Combining Theorem \ref{thm:hl} with Theorem \ref{thm:qsg}, we obtain.
	
	\begin{corollary}\label{cor:openbound}
		There is a computable function $f$ mapping (codes for) open sets $A \subseteq \omega^{<\omega}$ to trees such that, if there is a Laver tree $T$ with $[T] \subseteq \open{A}$, then $f(A)$ is wellfounded and there is one such Laver tree computable in $0^{(2 \cdot \height(f(A)))}$.
	\end{corollary}
	There does not seem to be a similar bound for Hechler trees whose bodies are contained in open sets. On the other hand, there is a more concrete proof of the restriction of Corollary \ref{cor:openbound} to clopen sets which also applies to Hechler trees.
	
	\begin{corollary}
		Let $(A_0, A_1)$ be prefix-free codes for open sets with $\baire \setminus \open{A_i}=\open{A_{1-i}}$, let $T_{A_i}$ be the wellfounded tree obtained closing $A_i$ under initial segments, and let $\alpha=\max\{\height(T_{A_0}), \height(T_{A_1})\}$. There is some $i \in \{0,1\}$ and a Laver tree $T$ with $[T] \subseteq \open{A_i}$ and $T \leq_{\T} 0^{(2 \cdot \alpha)}$.
	\end{corollary}
	\begin{proof}
		By Theorem \ref{thm:hl}, we know that there are $i \in \{0,1\}$ and a Laver tree $T$ such that $[T] \subseteq \open{A_i}$. Let $T'=\{\sigma \in T : \forall \tau \prec \sigma \,\, (\tau \notin A_i)\}$.
		It is easy to see that $[T] \subseteq \open{A_i}$ implies that $T'=T \cap T_{A_i}$. Consequently $T'$ is wellfounded and, for every $\sigma \in T'$, the tree rank of $\sigma$ in $T'$ is less than or equal to the tree rank of $\sigma$ in $T_{A_i}$. It is not hard to see that every string $\sigma \in T'$ appears in the ranking $(R_{\alpha}(A_i))_{\alpha \in \omega_1}$ defined as in the proof of Theorem \ref{thm:hl}, starting from $A_i$, and the tree rank of $\sigma$ in $T'$ majorizes its rank according to $(R_{\alpha}(A_i))_{\alpha \in \omega_1}$.
		
		This implies that $\langle \rangle \in R_{\beta}(A_i)$ for some $\beta \leq \height(T') \leq \height(T_{A_i})$. Again using the proof of Theorem \ref{thm:hl}, this means that there is a Laver tree $S \leq_{\T} 0^{(2 \cdot \beta)}$ (potentially different from $T'$) with $S \subseteq \open{A_i}$.
	\end{proof}
	Essentially the same proof establishes a similar bound for Hechler trees, namely: if $A$ is an clopen set and there is a Hechler tree $T$ with $[T] \subseteq \open{A}$, then there is a Hechler tree $S$ with $[S] \subseteq \open{A}$ and $S \leq_{\T} 0^{(\height(T_{A}))}$.
	
	\section{The reverse mathematics of the Laver partition theorem}\label{sec:rm}
	
	All the principles introduced are fairly strong, lying around the level of $\atr$. Working towards reversals, we start by showing that $\lld$, the weakest of all our principles, implies $\aca$.
	
	\begin{lemma}\label{lem:bootstrap}
		Over $\rca$, $\lld$ implies $\aca$.
	\end{lemma}
	\begin{proof}
		This proof is an elaboration on Jockusch's proof that Ramsey's Theorem for triples implies $\aca$ over $\rca$. Indeed, we use the exact same coloring.
		
		Fix an injective function $f \colon \mathbb{N} \rightarrow \mathbb{N}$. It suffices to show in $\rca + \lld$ that $\ran(f)$ exists. Let $c \colon [\mathbb{N}]^3 \rightarrow \{0,1\}$ be defined as:
		\begin{equation*}
			c(l, m, n)=
			\begin{cases}
				0 &\mbox{if } \forall i \leq l \,\, (\exists j \leq m \,\, (f(j)=i) \leftrightarrow \exists j \leq n \,\, (f(j)=i)) \\
				1 &\mbox{otherwise }
			\end{cases}
		\end{equation*} 
		and define the $\bDelta{0}{1}$ set $D=\{f \in \baire : c(f \upto 3)=0\}$.
		
		We claim that if $L$ is a Laver tree, then $[L] \cap D \neq \emptyset$. Indeed, let $l_0= \min \{n : \langle n \rangle \in L\}$. By bounded $\bSigma{0}{1}$-comprehension, the set $X=\{n  \leq l_0 : \exists j \,\, (f(j)=n)\}$ exists. By $\mathsf{B}\bSigma{0}{1}$ there exists $M$ such that for every $x \in X$ there exists $j$ with $f(j)=x$ and $j \leq M$. Since $L$ is Laver, there is some $m_0 > M$ such that $\langle l_0,m_0 \rangle \in L$. It is easy to see that, for every $g \in [L]$ if $\langle l_0, m_0 \rangle \prec g$, then $g \in D$.
		
		Now, by $\lld$, there must be a Laver tree $L$ such that $[L] \subseteq D$. We can see that for every $n \in \mathbb{N}$, 
		\[
		\exists j (f(j)=n)
		\leftrightarrow \forall \sigma \in L \,\,((|\sigma| \geq 2 \land \sigma(0) \geq n) \rightarrow \exists j \leq \sigma(1) \,\, f(j)=n)),
		\]
		so that $\ran(f)$ exists by $\bDelta{0}{1}$-comprehension.
	\end{proof}
	
	We now show that the principle $\hld$ implies $\atr$ over $\rca$. The proof of this is a modification of the proof that $\bDelta{0}{1}$-$\mathsf{RT}$ implies $\atr$.
	A key step for our reversal is the following result of \cite{miller}, which strengthens the tie between Hechler and Laver trees on the one hand, and the Galvin-Prikry theorem on the other.
	
	\begin{lemma}\label{lem:containsXo}
		The following is provable in $\rca+\mathsf{B}\bSigma{0}{2}$. For every Hechler tree $T$ there is some $X \in [\mathbb{N}]^{\mathbb{N}}$ such that $[X]^{\mathbb{N}} \subseteq [T]$.
	\end{lemma}
	\begin{proof}
		We want to obtain an increasing sequence $f \in \mathbb{N}^{\mathbb{N}}$ such that $[\ran(f)]^{\mathbb{N}} \subseteq [T]$. We build $f$ by recursion: we start with $\sigma_0=\langle \rangle$. Now given $n \in \mathbb{N}$, assume that we have built a strictly increasing sequence $\sigma_n$ such that for every $\tau \in [\mathbb{N}]^{<\mathbb{N}}$, if $\tau$ is an increasing enumeration of a subset of $\ran(\sigma_n)$, then $\tau \in T$. We let $\sigma_{n+1}=\concat{\sigma_n}{k}$ where $k$ is least such that, for every $\tau$ as above, $\concat{\tau}{k} \in T$. To see that such a $k$  must exist note that, since $T$ is a Hechler tree, for every $\tau$ enumerating a subset of $\sigma_n$, there are finitely many $j$ such that $\concat{\tau}{j} \notin T$. There are only finitely many sequences $\tau$ as above so, since by Propoaition \ref{prop:fuf} $\rca$+$\mathsf{B}\bSigma{0}{2}$ proves that finite unions of finite sets are finite, there must be an acceptable $k$. With these definitions, we let $f= \bigcup_{n \in \mathbb{N}}\sigma_n$. Since $f$ is strictly increasing, it enumerates an infinite set $X$.
		
		Now to see that $[X]^{\mathbb{N}} \subseteq [T]^{\mathbb{N}}$ notice that, by construction, if $\tau$ is an increasing enumeration of a finite subset of $X$, then $\tau \in T$. This immediately implies that for every $Y \in [X]^{\mathbb{N}}$, $Y \in [T]$.
	\end{proof}
	
	\begin{remark}
		Lemma \ref{lem:containsXo} has an arguably simpler proof, based on the Galvin-Prikry theorem (more precisely, on $\SRT{1}$). Indeed, let $T$ be a Hechler tree and consider any $A \in \ramsey$. Since $T$ has cofinite splitting at every node, it is easy to construct, by recursion, some $B \in [A]^{\omega}$ such that $p_B \in [T]$. Since $[T]$ is a closed set, it is Ramsey, so this implies that there must be some $X$ such that $[X]^{\omega} \subseteq [T]$.
		
		Miller's proof has the advantage of dispensing with the Galvin-Prikry theorem, and indeed the end goal of \cite{miller} was to provide an alternative proof of this result based on Hechler and Laver trees. In our context, the advantage of Miller's proof is that it requires much weaker axioms.
	\end{remark}
	We can now prove that $\hld$ implies $\atr$. As mentioned above, the proof closely follows Simpson's proof that $\DRT{1}$ implies $\atr$ over $\rca$ (see \cite[Lemma V.9.6]{simpson}).
	\begin{lemma}\label{lem:hldatr}
		Over $\rca$, $\hld$ implies $\atr$.
	\end{lemma}
	
	\begin{proof}
		We reason in $\aca$, using Lemma \ref{lem:bootstrap}.
		
		By Proposition \ref{prop:atrwithtrees}, it suffices to show that given a double sequence of trees $(T^0_k, T^1_k)_{k \in \mathbb{N}}$ such that for every $k$, at most one of $T^0_k$ and $T^1_k$ is illfounded, there is some set $Z$ such that for every $i$, if $T^0_i$ is illfounded, then $i \in Z$, and if $T^1_i$ is illfounded, then $i \notin Z$.
		
		In the remainder of the proof, given any $f \in \mathbb{N}^{\mathbb{N}}$ and any $k \in \mathbb{N}$, we denote by $f^k$ the function given by $f^k(i)=f(k+i)$. Moreover, given any tree $T$ and any string $\sigma \in \mathbb{N}^{<\mathbb{N}}$, we write $\leftof{T}{\sigma}$ to denote the set $\{\tau \in T : |\tau|=|\sigma| \land \tau \leq \sigma\}$ where $\tau \leq \sigma$ is a shorthand for $\forall i < |\tau| \,\, \tau(i) \leq \sigma(i)$. Notice that if $\leftof{T}{\sigma} = \emptyset$, then $\leftof{T}{\rho} = \emptyset$ for every $\rho \succeq \sigma$.
		
		From our assumptions on the sequence $(T^0_k, T^1_k)$, by weak K\H{o}nig's Lemma we have that for every $f \in \mathbb{N}^{\mathbb{N}}$ there is some $\sigma \prec f$ such that at least one between $\leftof{T^0_k}{\sigma}$ and $\leftof{T^1_k}{\sigma}$ is empty.
		Given $m$ we say that $\sigma$ is a \emph{$m$-critical sequence} if for every $i \leq m$, at least one between $\leftof{T^0_i}{\sigma}$ and $\leftof{T^1_i}{\sigma}$ is empty.
		Given any $f$ and any $m$, we define the \emph{first $m$-critical sequence} of $f$, denoted $c(m,f)$ as the least prefix $\sigma$ of $f$ which is $m$-critical. 
		Recursively, we define sequences $f_n$ and $(\sigma_i(f))_{i \in \mathbb{N}}$ as $f_0=f^1$, and, for every $n$, $\sigma_n(f)=c(f(0),f_n)$ and $f_{n+1}=f_n^{|\sigma_{n}(f)|}$. It is not hard to see that $f=f(0)^\smallfrown g$, where $g$ is the concatenation of the $\sigma_i(f)$'s. Moreover, each $\sigma_{n}(f)$ is the first $f(0)$-critical sequence of $f_n$.
		
		Given $m$, we can define a function $s_m$ which maps $m$-critical sequences $\sigma$ to $s_m(\sigma) \in 3^{m+1}$, as follows: for every $k \leq m$,
		\begin{equation*}
			s_m(\sigma)(k)=
			\begin{cases}
				0 &\mbox{if } \leftof{T^0_k}{\sigma} = \emptyset \land \leftof{T^1_k}{\sigma} = \emptyset \\
				1 &\mbox{if } \leftof{T^0_k}{\sigma} \neq \emptyset \land \leftof{T^1_k}{\sigma} = \emptyset \\
				2 &\mbox{if } \leftof{T^0_k}{\sigma} = \emptyset \land \leftof{T^1_k}{\sigma} \neq \emptyset
			\end{cases}
		\end{equation*}
		
		We can now define a $\bDelta{0}{1}$ set $C=\{f \in \mathbb{N}^{\mathbb{N}} : s_{f(0)}(\sigma_0(f))=s_{f(0)}(\sigma_1(f))\}$.
		
		The key step here is that for every set $X \in [\N]^{\N}$, $[X]^{\N} \cap C \neq \emptyset$. This is because, using $f$ in place of $p_X$, and thinking of $f$ as $\concat{f(0)}{g}$ as above, the function $n \mapsto s_{f(0)}(\sigma_n(f))$ maps $\N$ to $3^{f(0)}$, so there must be $n < m$ such that $s_{f(0)}(\sigma_n(f))=s_{f(0)}(\sigma_m(f))$. Now the sequence $f'=\concat{f(0)}{\concat{{\sigma_n(f)}}{\concat{{\sigma_m(f)}}{g'}}}$, where $g'$ is the tail of $f$ after the last entry of $\sigma_m(f)$, is in $[X]^{\N}$, and by definition $f' \in C$. Now Lemma \ref{lem:containsXo} immediately implies that there is no Hechler tree $T$ with $[T] \cap C = \emptyset$.
		
		By $\hld$, there must be some Laver tree $S$ such that $[S] \subseteq C$. We use $S$ to define a separating set $Z$, as follows. For a given $i \in \N$, let $\sigma_i \in S$ be the sequence with least code such that $\sigma_i(0) \geq i$ and $\sigma_i^-$, defined as $\sigma_i^-(n)=\sigma_i(n+1)$ for every $n < |\sigma_i|-1$, is a $\sigma_i(0)$-critical sequence. Let $i \in Z$ if and only if $s_{\sigma_i(0)}(\sigma_i^-)(i)=1$. We claim that $Z$ is a separating set for the sequence $(T^0_k, T^1_k)_{k \in \N}$. Indeed, suppose $T^0_i$ is illfounded, and let $g \in [T^0_i]$. Suppose by contradiction that $i \notin Z$, so that $s_{\sigma_i(0)}(\sigma_i^-)(i) \neq 1$. Now let $\sigma'$ be such that $\concat{{\sigma_i}}{\sigma'} \in S$ and $\sigma'(n) \geq g(n)$ for every $n < |\sigma'|$ (such $\sigma'$ exists because $S$ is a Laver tree). We have $\leftof{T^0_i}{\sigma'} \neq \emptyset$, so that $s_{\sigma_i(0)}(\sigma')(i) = 1$. Therefore, given any branch $f \in [S]$ extending $\concat{{\sigma_i}}{\sigma'}$, we get $f \notin C$. This is a contradiction. The argument to show that if $T^1_i$ is illfounded, then $i \notin Z$ is symmetrical.
	\end{proof}

	We now give the best known lower bounds for the strengths of $\lld$ and $\lls$, which involve finitary $\bSigma{1}{1}$-$\mathsf{AC}_0$. 
	We start with the lower bound to for the strength of $\lld$, the proof of which is based on modification of the proof of \ref{lem:hldatr}.
	
	\begin{proposition}\label{prop:delta11}
		$\rca +$ finitary $\bSigma{1}{1}$-$\mathsf{AC}_0 + \lld$ proves $\bDelta{1}{1}\text{-}\mathsf{CA}_0$.
	\end{proposition}
	\begin{proof}
		By Lemma \ref{lem:bootstrap} we can again also use $\aca$, and we can exploit Proposition \ref{prop:doowithtrees}.
		
		We consider a double sequence of trees $(T^0_i, T^1_i)$ such that for every $i$, exactly one of $T^0_i$ and $T^1_i$ is illfounded. We define the clopen set $C$ as in the proof of Lemma \ref{lem:hldatr}.
		
		We now show that no Laver tree is disjoint from $C$. Let $T$ be a Laver tree and let $m$ be least such that $\langle m \rangle \in T$. By finitary $\bSigma{1}{1}$-$\mathsf{AC}_0$, let $(g_i)_{i \leq m}$ be such that $g_i \in [T^0_i \cup T^1_i]$ for every $i \leq m$. Since $T$ is a Laver tree, there is some $f \in [T]$ such that $f(0)=m$ and $f^1(k) \geq \max_{i \leq m}\{g_i(k)\}$ for every $k \in \N$. For every $i \leq m$ only one of $[T^0_i]$ and $[T^1_i]$ is illfounded, and $f^1$ majorizes one element in each of the closed sets $([T^0_i \cup T^1_i])_{i \leq m}$. In the terminology of the proof of Lemma \ref{lem:hldatr}, the function $s_m$ is constant on the $m$-critical sequences that make up $f^1$. Hence $f \in C$, and $[T] \cap C \neq \emptyset$.
		
		By $\lld$, it follows that there is some Laver tree $T$ such that $[T] \subseteq C$. We define $Z$ exactly as in the proof of Lemma \ref{lem:hldatr}: by the hypothesis on our double sequence we must have $Z=\{i \in \N: [T^0_i] \neq \emptyset\}$.
	\end{proof}
	
	The Laver partition theorem for open sets is at least as strong as $\bSigma{1}{1}$-$\mathsf{AC}_0$.
	
	\begin{proposition}\label{prop:sigma11}
		$\rca +$ finitary $\bSigma{1}{1}$-$\mathsf{AC}_0 + \lls$ proves $\bSigma{1}{1}\text{-}\mathsf{AC}_0$.
	\end{proposition}
	\begin{proof}
		Again, by Lemma \ref{lem:bootstrap} we can use $\aca$. In particular, we have access to K\H{o}nig's Lemma and to Proposition \ref{prop:soowithtrees}.
		
		Let $(T_i)_{i \in \N}$ be a sequence of illfounded trees. We define an auxiliary sequence $(T^+_i)_{i \in \N}$ as follows: $T^+_i=\{\sigma \in \N^{<\N} : \exists \tau \in \N^{|\sigma|} \cap T_i \,\, \tau \leq \sigma\}$. For every $i$, the tree $T^+_i$ is the least \emph{smooth} tree containing $T_i$. It is known that $\wkl$ shows that $f \in [T^+_i]$ if and only if there is some $g \in [T_i]$ with $g \leq f$ (see \cite[Exercise VI.1.8]{simpson}).
		
		Now we define the set \[C=\{f \in \N^{\N} : \forall i \leq f(0) \,\, (\leftof{T_i}{f^2 \upto f(1)} \neq \emptyset \rightarrow  \exists g \in [T_i]\,\, g \leq f^2)\}.\] It is not hard to see (using K\H{o}nig's Lemma) that $C$ is closed.
		
		Now let $T$ be any Laver tree, let $m$ be the least such that $\langle m \rangle \in T$, and by finitary $\bSigma{1}{1}$-$\mathsf{AC}_0$ let $(g_i)_{i \leq m}$ be such that $g_i \in [T^+_i]$ for every $i \leq m$. Since $T$ is a Laver tree, there is some $f \in [T]$ with $f(0)=m$ and $f^2(k) \geq \max_{i \leq m}\{g_i(k)\}$ for every $k$. Since $f^2$ majorizes an element of $[T_i]$ for every $i \leq m$, it follows immediately that $f \in C$, hence $[T] \cap C \neq \emptyset$.
		
		Hence, $\lls$ guarantees that there is a Laver tree $T$ with $[T] \subseteq C$. 
		
		Given any $i$, we look for the least $n_0 \geq i$ such that $\langle n_0 \rangle \in T$ and the least $n_1$ such that $\langle n_0, n_1 \rangle \in T$. Then we look for the string $\sigma$ with least index such that $\concat{\langle n_0, n_1 \rangle}{\sigma} \in T$, $|\sigma| = n_1$ and $\leftof{T_i}{\sigma} \neq \emptyset$. We then set $h_i$ to be the leftmost path in $T$ extending $\concat{\langle n_0, n_1 \rangle}{\sigma}$. Note that, since $h_i^2 \upto h_i(1)=\sigma$, $\leftof{T_i}{h_i^2 \upto h_i(1)} \neq \emptyset$. Since $h_i \in [T] \subseteq C$, there must be some $g_i \in [T_i]$ with $g_i \leq h_i$.

        Given the sequence $(h_i)_{i \in \N}$, $\wkl$ suffices to prove the existence of a sequence of infinite paths $(g_i)_{i \in \N}$ with $g_i \in [T_i]$ for every $i$, as in the statement of $\bSigma{1}{1}$-$\mathsf{AC}_0$.
	\end{proof}

	\section{The Weihrauch degrees related to the Laver partition theorem}\label{sec:Wd}
	
	In this section we compare the Weihrauch degrees of functions naturally associated to the Laver partition theorem with common benchmarks around the level of $\atr$. Following the approach of \cite{marconevalenti} we can define several functions associated to the Laver partition theorem.
	
	\begin{definition}\label{def:basic}
		For any $\Gamma \in \{\bSigma{0}{1}, \bPi{0}{1}, \bDelta{0}{1}\}$, we define the following:
		\begin{itemize}
			\item $\mathsf{wFindL}_{\Gamma} \subc \Gamma(\baire) \multif \baire$ mapping $A$ to $\{f : [\LL(f)] \subseteq A\}$, with \[\dom(\mathsf{wFindL}_{\Gamma}) = \{A \in \Gamma(\baire): \forall T \,\, (T \text{ is a Laver tree} \rightarrow \exists g \in [T] \,\, g \in A)\},\]
			\item $\mathsf{FindL}_{\Gamma} \subc \Gamma(\baire) \multif \baire$ mapping $A$ to $\{f : [\LL(f)] \subseteq A\}$, with \[\dom(\mathsf{FindL}_{\Gamma}) = \{A \in \Gamma(\baire): \exists T \,\, (T \text{ is a Laver tree} \land [T] \subseteq A)\}.\]
		\end{itemize}
		Moreover, we use the symbols $\lls$ and $\lld$ for the functions most naturally associated to, respectively, the open and clopen Laver partition theorem. Namely, we let:
		
		\begin{itemize}
			\item $\lld \colon \bDelta{0}{1}(\baire) \multif \baire$ mapping $A$ to $\{f : [\LL(f)] \subseteq A \lor [\LL(f)] \cap A = \emptyset\}$, and
			\item $\lls \colon \bSigma{0}{1}(\baire) \multif \baire$ mapping $A$ to $\{f : [\LL(f)] \subseteq A \lor [\LL(f)] \cap A = \emptyset\}$.
		\end{itemize}	
		\end{definition}
		
		\begin{remark}\label{rem:nolaverhechler}
			     Note that if a tree $V$ codes a closed set $A$, then $A \in \dom(\lindp)$ exactly when there exists a Laver tree $S$ with $S \subseteq V$; moreover, by $\hlp$, $A \in \dom(\wlindp)$ exactly when
		there exists a Hechler tree $S$ with $S \subseteq V$.

		\end{remark}

        We will show that $\lls$ follows a recurring theme in the Weihrauch analysis of principles around $\atr$: $\lindp$ is equivalent to $\CBaire$, while $\wlinds$ is equivalent to $\UCBaire$. This pattern occurs also in the analysis of the open Ramsey theorem, the perfect tree theorem, and open determinacy (see \cite{marconevalenti} and \cite{kmp}). This phenomenon corresponds to the template of the proofs of these results in $\atr$, where one attempts to build a suitable hierarchy: a success (obtained using $\UCBaire$) yields to a solution to the open case, whereas a failure of the attempt leads to a pseudohierarchy (chosen using $\CBaire$), which in turn produces a solution to the closed case.
	
		\subsection{The one-sided principles}
		
		We now characterize the (arithmetical) Weihrauch degrees of the one-sided principles.

	\begin{theorem}\label{thm:ucbaire}
		The following functions are strongly Weihrauch equivalent to $\UCBaire$:
		\begin{itemize}
			\item $\wlindd$,
			\item $\wlinds$,
			\item $\lld$,
			\item $\lindd$,
			\item $\linds$.
		\end{itemize}
	\end{theorem}
	\begin{proof}
		We exploit the equivalences $\bDelta{1}{1}$-$\mathsf{CA} \equiv_{\sW} \mathsf{ATR} \equiv_{\sW} \UCBaire$ (see Proposition \ref{prop:ucbequiv}).
		
		Proposition \ref{prop:delta11} shows, essentially, that $\bDelta{1}{1}$-$\mathsf{CA} \lews \wlindd$. Note that all the other functions of our list are clearly $\geq_{\sW} \wlindd$.
		
		Now let $F$ be any of the five functions. We claim that $F \lews \bDelta{1}{1}$-$\mathsf{CA}$. Since the proofs corresponding to the five different functions are almost exactly the same (there are only tiny differences in the way instances are coded) we provide a high level proof sketch encompassing all the cases at once. 
		
		Note that by Corollary \ref{cor:openbound} there is a computable function $T$ such that, for every $f \in \dom(F)$, if $\alpha=|T(f)|$, then $F(f)$ has a member computable in $0^{(2 \cdot \alpha)}$. Moreover, in our context, this ordinal $2 \cdot \alpha$ is an upper bound on the rank of $\langle \rangle$ in the hierarchy $(R_{\alpha})_{\alpha \in \ock}$ based on the open set $\open{A}$ coded within $f$. It follows that one can obtain an element of $F(f)$ (uniformly in $f$) just by iterating a suitable arithmetical operator along the wellorder $\mathsf{KB}(T(f)) + \omega$. This implies that $F \lews \atrw$.
	\end{proof}
	
	We now show that the one-sided principles for closed sets are close to $\CBaire$.
	
	\begin{proposition}\label{prop:closedsol}
		$\lindp \equiv_{\sW} \CBaire$
	\end{proposition}
	\begin{proof}
		Let $S \subseteq \omega^{<\omega}$ be the code for an instance of $\lindp$. As in Definition \ref{def:codesl}, we define the set $C_S \subseteq C$, which is the set of (codes for) partial Laver trees contained in $S$. Then $C_S$ is an illfounded tree and for every $f \in [C]$ we have $f \in [C_S]$ if and only if $[\LL(J^{-1}(f))] \subseteq S$. Hence applying $\CBaire$ to $C_S$ yields a solution to the original instance $S$ for $\lindp$. Since $J^{-1}$ is computable, this is a strong Weihrauch reduction.
		
		For the other direction, given a tree $T$ which is an instance of $\CBaire$, let $T'=\{\sigma \in \omega^{<\omega} : \sigma_1 \in T\}$, where we consider the pairing function on strings obtained applying that on numbers componentwise, i.e.\ $\sigma_1(j)=(\sigma(j))_1$. Then $[T'] \in \dom(\lindp)$: indeed, if $f \in [T]$, the subtree of $T'$ given by $\{\sigma : \sigma_1 \prec f\}$ is a Laver tree. Given any Laver tree $S$ with $[S] \subseteq [T']$ and any $g \in [S]$, we have that $g_1 \in [T]$. This shows that $\CBaire \lews \lindp$.
	\end{proof}
	
	\begin{proposition}
		$\CCantor \star \wlindp \equiv_{\W} \CBaire$, so  $\wlindp \equiv^{\ari}_{\W} \CBaire$.
	\end{proposition}
	\begin{proof}
		The chain of reductions $\CCantor \star \wlindp \lew \CCantor \star \lindp \lew \CBaire$ follows from Proposition \ref{prop:closedsol} together with the fact that $\CBaire$ is closed under composition (see Proposition \ref{prop:summary}). A simple elaboration on the proof of Proposition \ref{prop:sigma11} shows the reduction $\CBaire \lew \CCantor \star \wlindp$.
	\end{proof}

	\subsection{More on finding Laver trees in closed sets}
	
	We show that $\wlindp \not\equiv_{\sW} \lindp$, as the former is not a cylinder. To achieve this, we show that $\lindp$ and $\wlindp$ are equivalent to restrictions of $\CBaire$ to smaller domains, in particular $\CBaire \upto \dom(\lindp) \equiv_{\sW} \lindp$ and $\CBaire \upto \dom(\wlindp) \equiv_{\sW} \wlindp$.
	
	The main observation is the following Lemma, which mentions the set of codes for partial Laver trees $C$, the decoding function $T$ and the isomorphism $j \colon \omega^{<\omega} \rightarrow C$ from Definition \ref{def:codesl}.
	
	\begin{lemma}\label{lem:treesofcodes}
		Let $U$ be a tree in $\dom(\lindp)$, and let $C_U=\{\sigma \in C : T(\sigma) \subseteq U\}$. Then $C_U, j^{-1}(C_U) \in \dom(\lindp)$. Moreover, if $U \in \dom(\wlindp)$, then also $j^{-1}(C_U) \in \dom(\wlindp)$.
	\end{lemma}
	\begin{proof}
		Assume that $U$ contains a Laver tree $S$. We build a Laver tree $S' \subseteq C_U$ recursively. Clearly $\langle \rangle \in S'$, and, if $\sigma \in S'$, then we add to $S'$ the set $\{\concat{\sigma}{n} : T(\concat{\sigma}{n}) \subseteq S\}$. We get that $S'$ is a Laver tree as, inductively, $T(\sigma) \subseteq S$, and there are infinitely many possible ways to extend $T(\sigma)$ by adding exactly one child to every node, picking these children in $S$. The fact that $j^{-1}(S') \subseteq j^{-1}(C_U)$ is a Laver tree follows from the fact that $j$ is an isomorphism of trees.
		
		Now assume that $S$ were an Hechler tree (recall Remark \ref{rem:nolaverhechler}). It is not hard to see that, for every $\sigma \in S'$, there are only finitely many immediate successors of $\sigma$ in $C$ which are not in $S'$. In other words, $j^{-1}(S') \subseteq j^{-1}(C_U)$ is a Hechler tree.
	\end{proof}
	
	\begin{proposition}\label{prop:noteq} $\CBaire \upto \dom(\lindp) \equiv_{\sW} \lindp$ and $\CBaire \upto \dom(\wlindp) \equiv_{\sW} \wlindp$
	\end{proposition}
	\begin{proof}
		The reduction $\CBaire \upto \dom(\lindp) \leq_{\sW} \lindp$ follows from Proposition \ref{prop:closedsol}. To show that $\CBaire \upto \dom(\wlindp) \leq_{\sW} \wlindp$ notice that, given a tree $T \in \dom(\wlindp)$, since any $S \in \wlindp(T)$ is a pruned tree contained in $T$, one can simply take its leftmost path to obtain a solution to $\CBaire(T)$.
		
		For the converse reductions we use Lemma \ref{lem:treesofcodes}. Let $U \in \dom(\lindp)$ and consider $j^{-1}(C_U)$. We know that $j^{-1}(C_U) \in \dom(\lindp)$, and, if $f$ is any path in $[j^{-1}(C_U)]$, then $\LL(f)$ is a Laver tree with $\LL(f) \subseteq U$. The proof of $\wlindp \leq_{\sW} \CBaire \upto \dom(\wlindp)$ is identical.
	\end{proof}
	
	To show that $\wlindp$ is not a cylinder, we rely on the following fact from classical computability theory (see \cite{hypencsets}).
	
	\begin{theorem}[Jockush-Solovay]\label{thm:recenc}
		A set $Z \subseteq \omega$ is hyperarithmetical if and only if, for every $X \in \ramsey$, there is some $Y \in [X]^{\omega}$ such that $Z \leq_{\T} Y$.
	\end{theorem}
	
	We now present a tool for Weihrauch analysis based on Jockush and Solovay's characterization. A similar tool has been previously used by the authors in \cite{marconeosso1}.
	
	\begin{lemma}\label{lem:nocylinder}
		Let $f \subc \baire \multif \baire$ be such that, for every $x \in \dom(f)$ and every $X \in \ramsey$, $[X]^{\omega} \cap f(x) \neq \emptyset$. Then $\id{{\baire}} \nleq_{\sW} f$, and in particular $f$ is not a cylinder.
	\end{lemma}
	
	An $f$ as above was said to ``meet every infinite set'' in \cite[Proposition 5.22]{marconeosso1}.
	
	\begin{proof}
		Suppose $(\varphi, \psi)$ witnesses a reduction $\id{{\baire}} \leq_{\sW} f$, and let $x \in \baire$ be a non-hyperarithmetical real. We obtain that, for every $g \in f(\varphi(x))$, $\psi(g)=x$. Since $f$ meets every infinite set, it follows that for every $X \in \ramsey$ there is some $Y \in [X]^{\omega}$ such that $x \leq_{\T} Y$. By Theorem \ref{thm:recenc}, this contradicts the fact that $x$ is not hyperarithmetical.
	\end{proof}
	
	\begin{theorem}
		$\wlindp$ is not a cylinder, hence $\wlindp \not\equiv_{\sW} \lindp$.
	\end{theorem}
	\begin{proof}
		Since ``being a cylinder'' is preserved under strong Weihrauch equivalence, we show that $\CBaire \upto \dom(\wlindp)$ is not a cylinder. By Lemma \ref{lem:nocylinder}, it suffices to show that $\CBaire \upto \dom(\wlindp)$ meets every infinite set. This follows from the fact that, if $T$ is a Hechler tree and $X \in \ramsey$, there is some $Y \in [X]^{\omega}$ such that $p_Y \in T$ (where $p_Y$ is the principal function of $Y$).
	\end{proof}
	
	We conclude the section by showing that $\wlindp$ is parallelizable. Note that $\lindp$ and all the principles mentioned in Theorem \ref{thm:ucbaire} are parallelizable because $\CBaire$ and $\UCBaire$ are.
	
	\begin{proposition}
		$\wlindp \equiv_{\W} \widehat{\wlindp}$.
	\end{proposition}
	\begin{proof}
		Let $((C_m)_{m \in \omega})$ be an instance of $\widehat{\wlindp}$. We define a closed set $C$ as follows: $f \in C$ if and only if $f \in \bigcap_{i \leq f(0)}C_i$. To prove that $C \in \dom(\wlindp)$ let $T$ be any Laver tree. Let $m$ be least such that $\langle m \rangle \in T$ and let $S=\{\sigma \in \omega^{<\omega} : \concat{m}{\sigma} \in T\}$. Notice that $S$ is a Laver tree. Consider the closed sets $(C_n)_{n \leq m}$: since each of these belongs to $\dom(\wlindp)$, there are Hechler trees $(S_n)_{n \leq m}$ such that, for every $n$, $[S_n] \subseteq C_n$. Therefore, the Hechler tree $S'=\bigcap_{n \leq m} S_n$ is such that $[S'] \subseteq \bigcap_{n \leq m} C_n$. Since $[S] \cap [S]' \neq \emptyset$, it follows that $[S] \cap \bigcap_{n \leq m} C_n \neq \emptyset$ and hence $[T] \cap C \neq \emptyset$. This shows that $C \in \dom(\wlindp)$.
		
		Given $S \in \wlindp(C)$ it is immediate to build a solution $(S_m)_{m \in \omega}$ to the  $\widehat{\wlindp}$ instance $((C_m)_{m \in \omega})$.
	\end{proof}
	
	\subsection{The degree of \texorpdfstring{$\mathsf{LL}_{\bSigma{0}{1}}$}{LLS}}
	
	We conclude our analysis with a study of the position of $\lls$ in the (arithmetical) Weihrauch degrees.
	
	\begin{lemma}\label{lem:abovetcb}
		$\tCBaire \leq_{\W} \CCantor \star \lls$, so $\tCBaire \leq^{\ari}_{\W} \lls$.
	\end{lemma}
	\begin{proof}
		This is an elaboration on the proof of Proposition \ref{prop:sigma11}, using the equivalence $\CCantor \equiv_{\W} \tCCantor$. Given a tree $T$, thought of as an instance of $\tCBaire$, we consider the tree $T^+$. Note that if $[T] \neq \emptyset$, then $T^+ \in \dom(\wlindp)$, hence any $S \in \lls(T^+)$ is such that $S \subseteq T^+$. In particular, in this case we would have that the leftmost path $f \in [S]$ is a bound for some path in $T$, so that $\leftof{T}{f} \in \dom(\CCantor)$ and any $g \in \CCantor(\leftof{T}{f})$ is a path in $[T]$. If $[T]=\emptyset$, the same algorithm described above still works: given any Laver tree $S \in \lls(T^+)$, we take its leftmost path and notice that any $g \in \tCCantor(\leftof{T}{f})$ is a valid solution to $\tCBaire(T)$. 
	\end{proof}
	
	We immediately obtain a separation between the Weihrauch degree of $\CBaire$ and that of $\lls$, as in \cite[Corollary 4.15]{marconevalenti}
	
	\begin{corollary}\label{cor:abovewf}
		$\mathsf{WF} \leq_{\W} \mathsf{LPO} \star \lls$, so $\lls \nleq_{\W} \CBaire$.
	\end{corollary}
	\begin{proof}
		The reduction follows from the proof of Lemma \ref{lem:abovetcb}: given a tree $T$, the leftmost path $g$ of any $S \in \lls(T^+)$ is guaranteed to be a bound for an element of $[T]$, if the latter is nonempty. Using $\mathsf{LPO}$ we can check whether, for every $n$, there is some $\sigma \in T \cap \omega^n$ such that $\sigma(i) \leq g(i)$ for all $i<|\sigma|$, and hence obtain the solution to $\mathsf{WF}(T)$. The separation $\lls \nleq_{\W} \CBaire$ follows as in the proof of \cite[Corollary 4.15]{marconevalenti} from $\mathsf{WF} \nleq_{\W} \CBaire$ (see \cite[Section 7]{kmp}).
	\end{proof}
	
	We show that the arithmetical reduction of Lemma \ref{lem:abovetcb} is strict, closely following the proof of \cite[Corollary 5.13]{marconevalenti}. The first ingredient is a variant of Corollary \ref{cor:abovewf}, in which we show that the reduction $\mathsf{WF} \leq_{\W} \mathsf{LPO} \star \lls$ can be specialized to two particular domains.
	
	\begin{lemma}\label{lem:comptrees}
		Let $X \subseteq \dom(\lls)$ consist of the instances of $\lls$ which have no arithmetical solutions, and let $Y$ be the set of computable trees. Then $\mathsf{WF} \upto Y \leq_{\W} \mathsf{LPO} \star \lls \upto X$.
	\end{lemma}
	\begin{proof}
		We follow the argument for the analogous reduction present in \cite[Theorem 4.20]{marconevalenti}. Let $\Phi$ be the forward functional of the reduction $\UCBaire \leq_{\W} \wlindd$, let $S$ be a computable instance of $\UCBaire$ with no arithmetical solution and let $T \in Y$. Reasoning as in the proof of \cite[Theorem 4.20]{marconevalenti} we may assume that $T$ has no hyperarithmetic path. We compute a name $p$ for the open set $(\baire \setminus [T^+]) \cap \delta_{\bDelta{0}{1}}(\Phi(S))$. We claim that $p \in X$, i.e.\ that $p$ has no arithmetical solution. There are two cases:
		\begin{itemize}
			\item if $[T]=\emptyset$, then $[T^+]=\emptyset$, so that $p$ is a name for $\delta_{\bDelta{0}{1}}(\Phi(S))$. Since any $\lls$ solution of $\delta_{\bDelta{0}{1}}(\Phi(S))$ computes the non-arithmetical path in $[S]$, we get that $p \in X$,
			\item if $[T] \neq \emptyset$, then $[T^+] \in \dom(\wlindp)$. It follows that $\delta_{\bSigma{0}{1}}(p)$ is in $\dom(\wlindp)$, and in particular any Laver tree $L \in \lls(\delta_{\bSigma{0}{1}}(p))$ has $[L] \subseteq \baire \setminus \delta_{\bSigma{0}{1}}(p) \subseteq [T^+]$. So the leftmost path of $L$, which is $L$-computable as $L$ is pruned, is a bound for a path in $T$. Since $T$ has no hyperarithmetic path, it follows that $L$ is non-hyperarithmetic.
		\end{itemize}
		This shows that $p \in X$, and moreover, that $[T] \neq \emptyset$ if and only if, given any $L$ in $\lls(\delta_{\bSigma{0}{1}}(p))$, the leftmost path of $L$ is a bound for some path in $T$. As in the proof of Corollary \ref{cor:abovewf} we obtain $\mathsf{WF} \upto Y \leq_{\W} \mathsf{LPO} \star \lls \upto X$.
	\end{proof}
	
	\begin{proposition}
		$\tCBaire <^{\ari}_{\W} \lls$.
	\end{proposition}
	\begin{proof}
		The reduction holds by Lemma \ref{lem:abovetcb}. By \cite[Theorem 5.12]{marconevalenti}, it suffices to show that $\lls \nleq_{\W} \mathsf{sTC}_{{\baire}} \times \lim^{(n)}$ for any $n \in \omega$. Again we prove this fact following the proof of the analogous \cite[Theorem 4.20]{marconevalenti}. Reasoning exactly as in the cited proof, we can see that $\lls \leq_{\W} \mathsf{sTC}_{{\baire}} \times \lim^{(n)}$ implies $\lls \upto X \leq_{\W} \CBaire$, with $X$ as in Lemma \ref{lem:comptrees}. This is a contradiction as $\mathsf{LPO} \leq_{\W} \CBaire$, $\CBaire$ is closed under composition and $\mathsf{WF} \upto Y \nleq_{\W} \CBaire$.
	\end{proof}
	\bibliographystyle{alpha}

\end{document}